\documentclass[12pt]{amsart}
\usepackage[utf8]{inputenc}
\usepackage[active]{srcltx}
\usepackage{hyperref}
\usepackage{tikz,pgfplots}
\pgfplotsset{compat=1.16} 

\usepackage{amsmath,amssymb,amsthm,mathtools}
\usepackage{mathabx}

\newtheorem{theorem}{Theorem}
\newtheorem*{theorem*}{Theorem}
\newtheorem{corollary}{Corollary}
\newtheorem{lemma}{Lemma}
\newtheorem{proposition}{Proposition}
\newtheorem*{proposition*}{Proposition}

\theoremstyle{definition}
\newtheorem{remark}[theorem]{Remark}
\newtheorem{remarks}[theorem]{Remarks}
\newtheorem{definition}[theorem]{Definition}

\newtheorem*{questions*}{Questions}
\newtheorem*{question*}{Question}
\newtheorem*{notation*}{Notation}

\def\R{\mathbb{R}}

\def\Z{\mathbb{Z}}
\def\N{\mathbb{N}}

\def\T{\mathbb{T}}

\def\cA{{\mathcal A}}
\def\cC{{\mathcal C}}
\def\cD{{\mathcal D}}
\def\cE{{\mathcal E}}
\def\cF{{\mathcal F}}
\def\cL{{\mathcal L}}
\def\cJ{{\mathcal J}}
\def\cM{{\mathcal M}}
\def\cN{{\mathcal N}}
\def\cQ{{\mathcal Q}}
\def\cR{{\mathcal R}}
\def\cS{{\mathcal S}}
\def\cT{{\mathcal T}}
\def\cU{{\mathcal U}}
\def\cV{{\mathcal V}}
\def\cW{{\mathcal W}}
\def\cZ{{\mathcal Z}}

\def\e{\epsilon}

\def\ent{\mathrm{ent}\,}

\def\logvol{\mathrm{logvol}\,}
\def\rad{\mathrm{rad}\,}
\def\rank{\mathrm{rank}\,}
\def\vol{\mathrm{vol}\,}

\begin{document}

\title[Invariant submanifolds]{Invariant Submanifolds\\
  of conformal Symplectic Dynamics} 

\author[M.-C. Arnaud \& J. Fejoz]{Marie-Claude Arnaud$^{\dag,\ddag,\circ}$ \& Jacques Fejoz$^{*,**,\circ}$}

\email{Marie-Claude.Arnaud@imj-prg.fr, jacques.fejoz@dauphine.fr}

\date{}

\keywords{ conformal symplectic dynamics, isotropy,
  entropy, exactness, Lagrangian submanifold, invariant manifold.}

\subjclass[2020]{ 37C05,37J39, 38A35      	 }

\thanks{$\dag$ Universit\'e de Paris and Sorbonne Universit\'e, CNRS, IMJ-PRG, F-75006 Paris, France. }  
\thanks{$\ddag$ Member of the {\sl Institut universitaire de France}}
\thanks{$\circ$ ANR AAPG 2021 PRC CoSyDy: Conformally symplectic dynamics, beyond symplectic dynamics}
\thanks{$*$ Université Paris Dauphine -- PSL, Ceremade}
\thanks{$**$ Observatoire de Paris -- PSL, IMCCE}

\maketitle

\begin{abstract}
We study invariant manifolds of conformal symplectic dynamical systems on a symplectic manifold $(\cM,\omega)$ of dimension $\geq 4$. This class of systems is the $1$-dimensional extension of symplectic dynamical systems for which the symplectic form is transformed colinearly to itself. 

In this context, we first examine how the $\omega$-isotropy of an invariant manifold $\cN$ relates to the entropy of the dynamics it carries. Central to our study is Yomdin's inequality, and a refinement obtained using {  that the local entropies have no effect transversally to the characteristic foliation of $\cN$.} 

When $(\cM,\omega)$ is exact and $\cN$ is isotropic, 
we also show that $\cN$ must be exact for some choice of the primitive of $\omega$, under the condition that the dynamics acts trivially on the cohomology of degree $1$ of $\cN$. The conclusion partially extends to the case when $\cN$ has a compact one-sided orbit. 

We eventually prove the uniqueness of invariant submanifolds $\cN$ when $\cM$ is a cotangent bundle, provided that the dynamics is isotopic to the identity among Hamiltonian diffeomorphisms. In the case of the cotangent bundle of the torus, a theorem of Shelukhin allows us to conclude that $\cN$ is unique even among submanifolds with compact orbits.
\end{abstract}

\tableofcontents

\section{Introduction} \label{SecShortIntro}


Let $(\cM^{2d}, \omega)$ be a symplectic manifold. Symplectic dynamical systems (so-called conservative dynamical systems) form a class of infinite codimension. We will study conformal symplectic dynamics, a now classical extension of symplectic dynamics\footnote{Vaisman~\cite{Vaisman1985} and others have defined local conformal symplectic structures on a manifold $\cM$. There is a corresponding notion of dynamics preserving the structure, thus extending our setting.} where the symplectic form may change in its own direction:

\begin{definition}
  \quad
  \begin{itemize} 
  \item A diffeomorphism $f: \cM \righttoleftarrow$ is \emph{conformal symplectic} if $f^*\omega=a\, \omega$ for some $a>0$ (conformality ratio).\footnote{As Libermann noticed
      \cite{Lib1959}: if $f^*\omega=a\, \omega$ for some smooth function
      $a$, $a\,\omega$ being closed we have $da \wedge \omega = 0$, which
      implies, if $\cM$ has dimension $\geq 4$, that $a$ is constant.}
  \item A complete vector field $X$ on $\cM$ is \emph{conformal symplectic} if $L_X\omega=\alpha\, \omega$, where $L_X$ is the Lie derivative, for some $\alpha \in \R$ (conformality rate).\footnote{Then the flow $(\varphi_t)$ of $X$ is conformal symplectic and $\varphi_t^*\omega=e^{\alpha t}\omega$.}
  \end{itemize}
\end{definition}

Such dynamics encapsulate mechanical systems whose friction force is proportional to velocity, in which case $a<1$ or $\alpha<0$.

In this paper we will focus on the non-symplectic case, i.e. $a \neq 1$ and $\alpha \neq 0$. Of course, time reversal changes $a$ in $1/a$ and $\alpha$ in $-\alpha$.

For such a dynamics, the volume form $\omega^{\wedge d}$ is monotonic. So if such a dynamics exists on $\cM$, $\cM$ cannot be closed and has infinite volume. Moreover, when the dynamics is given by a vector field $X$, the symplectic form satisfies $\omega=\frac{1}{\alpha} L_X\omega= d\left(\frac{1}{\alpha}i_X\omega\right)$ and is exact. Hence conformal vector fields exist only on exact symplectic manifolds. Yet this is not the case for conformal diffeomorphisms (see an example in Proposition \ref{Pnonisotrop}).

Also, if a vector field $X$ is conform symplectic of conformality rate $\alpha$ and if $Z$ is the Liouville vector field associated with the $1$-form $\lambda= - \frac{1}{\alpha}i_x\omega$ i.e., $i_Z\omega = \lambda$, then $X+\alpha Z$ is symplectic. Thus conformal symplectic vector field form a $1$-dimensional extension of the space of symplectic vector fields. When $(\cM, \omega)$ is exact, there exists a 1-parameter subgroup $\cC$ of the set of conform symplectic diffeomorphims such that the group of conform symplectic diffeomorphisms is $\{f\circ g; (f, g)\in \cC\times\cS\}$ where $\cS$  is the set of symplectic diffeomorphisms. When $\cM$ is not exact, let $\cR$ be the subgroup of $\R_+^*$ of conformal ratios of conformal symplectic diffeomorphisms of $\cM$. This subgroup can be trivial, e.g. when $\cM$ is compact (all conform symplectic diffeomorphism are symplectic). 

\begin{questions*}
  Can $\cR$ be strictly between $\{1\}$ and $\R^*_+$? Assuming that $\cR=\R^*_+$, does there exist a continuous $1$-para\-meter family of conform symplectic diffeomorphisms indexed by its conformal ratio in $\R^*_+$?
\end{questions*}

\medskip An important case is that of cotangent bundles $(\cM = T^*\cQ, \omega=-d\lambda)$, where $\cQ$ is a manifold and $\lambda$ is the canonical Liouville 1-form. A continuous-time example is the flow $\exp(tZ_\lambda)(q, p)=(q, e^{-t}p)$ of the Liouville vector field $Z_\lambda$ defined by $i_{Z_\lambda}(-d\lambda)=\lambda$ and
a discrete-time example is $f = \exp Z_\lambda: (q, p) \mapsto (q, a p)$, $a=e^{-1}$. These two examples of conformal symplectic dynamics have a very simple behaviour: 
\begin{itemize}
    \item there is a global attractor $\cA$;
    \item the $\omega$-limit set of every orbit is a point of  $\cA$.
\end{itemize}
More generally, consider a discounted Tonelli vector field $X$ on $T^*\cQ$ of negative rate $\alpha$; by definition it satisfies $i_X\omega = dH + \alpha \lambda$ for some Hamiltonian $H$  which is superlinear in the fiber direction and whose Hessian in the fiber direction is positive definite. It has been shown that the flow of such a vector field has a global attractor \cite{MarSor2017}. 

\medskip In the general setting, many natural questions are open, for example:

\begin{questions*}
    Which conditions ensure the existence of a global attractor? And provided that the global attractor exists (necessarily having zero volume), what can be said of its size?
\end{questions*}

\bigskip As a first step, in this article we focus on the case of invariant submanifolds (with a digression on the case of submanifolds with compact orbit), although the study of dissipative twist maps proves that there can exist invariant subsets that are not submanifolds \cite{LeCalvez1988}. 

\medskip First, we explore the isotropy of invariant submanifolds. This question is akin to its analogue in symplectic dynamics, where both negative and positive results have been proven in particular for invariant tori carrying minimal quasiperiodic flows. 

We start by providing an example where an invariant submanifold is a hypersurface and hence non-isotropic (Propositions \ref{Pflownonisotrop} and \ref{Pnonisotrop} in section \ref{Sisotropy}). There exist similar examples due to McDuff, \cite{McDuff1991} and Geiges \cite{Geiges1994,Geiges1995}, but our example is somewhat more explicit. We do not know if a similar example exists on a cotangent bundle. An even more difficult question is to determine whether such submanifolds may exist for discounted Tonelli flows on cotangent bundles. In this case and when $\dim \cM\geq 4$, the global attractor never separates $\cM$ and hence cannot be a hypersurface.

In turn, we show some positive results regarding the isotropy of invariant submanifolds. If the invariant submanifold is a surface, isotropy follows from a simple argument using the growth of the area. In higher dimension, a first result follows from Yomdin's theory \cite{Yom1987, Gromov1987}. Proposition \ref{PYomdin} of section \ref{Sisotropy} states that if a smooth\footnote{Smooth means $C^\infty$.} conformal diffeomorphism $f: \cM \righttoleftarrow$ with conformality rate $a$ has an invariant smooth submanifold $\cN \subset \cM$ such that the topological entropy of $f_{|\cN}$ is less than $|\log(a)|$, then $\cN$ is isotropic. 

But Yomdin's proof can be improved in the setting of diffeomorphisms which are conform with respect to a presymplectic form. Here, we prove that { the so-called local entropies have no effect on the volume growth transversally to the characteristic foliation of $\cN$ (section~\ref{Sentropy})}. It follows that if a conformal symplectic $C^3$-diffeomorphism
of conformality ratio $a$ has an invariant $C^3$-manifold on which $\omega$ has constant rank $2\ell$ and such that the entropy of $f_{|\cN}$ is smaller than $\ell\, |\log a|$, $\cN$ is isotropic. In particular, if an invariant submanifold carries a minimal dynamics (every orbit is dense) with zero entropy, it is isotropic (corollary~\ref{CCisoento}).

This new result assumes less regularity than the former one ($C^3$ instead of smooth in Proposition~\ref{PYomdin}) but requires that the symplectic form restricted to the submanifold has constant rank. 

A related result is \cite[2.2.1]{CaCeLlA2013}, where the authors prove that if a $C^1$ conformal dynamics has a $C^1$ invariant torus on which the dynamics is $C^1$ conjugate to a rigid rotation, then this torus is isotropic. This results is a direct consequence of Proposition \ref{PYomdin}. Corollary \ref{CCisoento} of section \ref{Sentropy} doesn't imply this result because our result require more regularity, and on the other hand our result applies when a $C^3$ dynamics is $C^0$ conjugated to a transitive rotation.

\medskip Second, we examine the question of exactness. In this purpose, in section~\ref{SLiouville} we assume that $(\cM, \omega=-d\lambda)$ is exact. Define the Liouville class of an isotropic embedding in $\cM$ as the cohomology class of the form induced by $\lambda$. The embedding is called exact when this class vanishes. 
The action of conform symplectic diffeomorphisms on Liouville classes depends on a notion of exactness for the diffeomorphisms themselves. Let $f: \cM \righttoleftarrow$ be a conformal symplectic diffeomorphism of ratio $a$. The form $f^*\lambda-a\lambda$ is closed.

\begin{definition}
    The diffeomorphism $f$ is \emph{$\lambda$ conformal exact symplectic (CES)} if $f^*\lambda-a\lambda$ is exact.
    It is \emph{Hamiltonian} if $f$ is the time-one map of the flow of a non autonomous conformal Hamiltonian vector field $X_t$ (meaning that $i_{X_t}\omega = \alpha_t \, \lambda + dH_t$ for all $t$).
\end{definition}

This definitions depend of the chosen primitive of the symplectic form. We prove in appendix \ref{Aconfexact} that there is always a choice of primitive for which $f$ is exact. Alternatively, we also show that $f$ is symplectically conjugate to a diffeomorphism which is 
exact with respect to the initial $\lambda$. Hence we state our results for exact conformal symplectic dynamics (see section \ref{SLiouville} for more comprehensive statements).

Our main result here is that if $f$ is an exact conform symplectic diffeomorphism and if $\cS$ is a strongly $f$-invariant submanifold (in the sense that $j \circ f (\cS) = j(\cS)$ and $f$ acts trivially on $H^1(j(\cS),\R)$), $j$ is exact. 

When $\cL$ is a Lagrangian submanifold that is H-isotopic\footnote{By \emph{H-isotopic}, we mean isotopic among Hamiltonian diffeomorphisms.} to a graph  in $\cM=T^*\cQ$ and $f$ is CS isotopic\footnote{By {\sl CS isotopic}, we mean isotopic among conform symplectic diffeomorphisms.} to ${\rm Id}_\cM$, we obtain the same conclusion when assuming only that the orbit of $\cL$ is bounded. For example, the submanifolds that are H-isotopic to the zero section and contained in an attractor satisfy this hypothesis.

\begin{question*} 
    Is it possible to obtain similar results without assuming that the Lagrangian submanifold is H-isotopic to a graph? On other manifolds?
\end{question*}

\medskip Third, in section 6, we raise the question of the uniqueness of a invariant Lagrangian submanifolds in a cotangent bundle $(T^*\cQ, -d\lambda)$. Indeed, let $f:T^*\cQ \righttoleftarrow$  be a CES diffeomorphism that is CH isotopic\footnote{By \emph{CH-isotopic}, we naturally mean isotopic among conformally Hamiltonian diffeomorphisms} to ${\rm Id}_{T^*\cQ}$. We show that there exists at most one submanifold of $T^*\cQ$ that is H-isotopic to the zero section and invariant by $f$. Key to the proof is the Viterbo distance of Lagrangian submanifolds which are H-isotopic to the zero section, and the fact that this distance is monotonic with respect to the action of $f$.

A recent result of Shelukhin even allows us to show the following. Let $f:T^*\T^n \righttoleftarrow$  be a CES diffeomorphism that is CH-isotopic to ${\rm Id}_{T^*\T^n}$. Then there exists at most one submanifold $\cL$ which is H-isotopic to the zero section and such that
$$\bigcup_{k\in\Z}f^k(\cL)\quad\text{is}\quad\text{relatively}\quad\text{compact}.$$
Hence when it exists, $\cL$ is invariant by $f$.

For discounted Tonelli flows, it was known that there is at most one invariant exact Lagrangian graph because this corresponds to the unique weak KAM solution \cite{MarSor2017}. But we give in Section \ref{AExamples} an example of such a dynamics with an invariant H-isotopic to a graph submanifold that is not a graph, hence even in this case our uniqueness result is new.

\section{Isotropy}\label{Sisotropy}

The so-called Ma{\~n}{\'e} example \cite{Man1992} (see subsection \ref{SecMan}) shows that any flow defined on a closed manifold $\cQ$ can be achieved as the restriction of a Tonelli conformal
Hamiltonian flow to the zero section of $T^*\cQ$. In this case, the zero section is an invariant Lagrangian submanifold.

The following example, which is very similar to an example of \cite{Geiges1995}, is key to this section. It  shows that a closed submanifold which is invariant by a conformal symplectic dynamics may be non $\omega$-isotropic. In the remaining of the section, we will give some general conditions under which the submanifold must be $\omega$-isotropic.

\begin{proposition}\label{Pflownonisotrop}
There exists a conformal symplectic vector field $X$ on a 4-dimensional symplectic manifold $({\cM}, \omega)$,with a
  3-dimensional invariant submanifold ${\cL}$ (hence ${\cL}$ is not
  isotropic). 
  
  Moreover, the submanifold ${\cL}$ is the global attractor for the flow $(\varphi_t)$ of $X$, $(\varphi_{t|\cL})$ is conjugate to the suspension of an Anosov automorphism of $\T^2$ with 2-dimensional stable and unstable foliations, and $(\varphi_{t|\cL})$ is transitive with entropy equal to $|\alpha|$, where $\alpha$ is the conformality rate of $X$.
\end{proposition}

\begin{remarks}
  \begin{enumerate}
    
  \item In our example, ${\cL}$ is coisotropic, but it is easy to extend this example to an invariant submanifold which is neither isotropic nor coisotropic. Indeed, let $Y$ be a conformal
    symplectic vector field on a symplectic manifold $(\cN, \omega')$  with a periodic orbit $\gamma$. Then the sum $X\oplus Y$ admits ${\cL}\times \gamma$ as an invariant submanifold that is neither isotropic nor coisotropic in $\cM\times\cN$ if $\dim\cN\geq 4$.
  
  \item The submanifold ${\cL}$ is the maximal (among compact subsets)
    attractor of the dynamics. 
    \item Replacing the vector field $X$ by $bX$ for $b\in\R$, we can achieve any positive value for the entropy.
  \end{enumerate}
\end{remarks}
\begin{questions*}
We don't know if it is possible to build a non-isotropic example on a cotangent bundle endowed with its usual symplectic form or, even stronger, if a similar example exists on such a manifold among Tonelli flows.
\end{questions*}
\begin{proof}[Proof of Proposition \ref{Pflownonisotrop}]We consider an Anosov  automorphism $A: \T^2\righttoleftarrow$ induced by a matrix $\begin{pmatrix}
a &b\\
c &d
\end{pmatrix} \in SL(2, \Z)$ with eigenvalues $0<\lambda_-< 1 < \lambda_+=\frac{1}{\lambda_-}$ and eigenvectors $v_\pm$. An example of such an automorphism is
$A(x,y)=(2x+y, x+y)$, with eigenvalues $\lambda_-=\frac{3-\sqrt{5}}{2}<1$ and
$\lambda_+=\frac{3+\sqrt{5}}{2}>1$. 

Following \cite{ArnoldAvez1967}, we define a suspension of the diffeomorphism $T$ by using the following relation on $\T^2\times \R$ (writing $\xi=(x,y)$):
$$\forall (\xi,z)\in\T^2\times \R, (\xi,z)\sim F(\xi,z) := (A\xi, z-1).$$
Denote by $\alpha_\pm$ the linear forms on $\R^2$ such that $\alpha_\pm(v\pm)=1$ and $\alpha_\pm(v_\mp)=0$. Observe that $\alpha_\pm\circ A=\lambda_\pm\alpha_\pm$. Rescale the forms $\alpha_\pm$ in the $z$-direction in order to get $F$-invariant forms on $\T^2\times \R$: define 
$$\beta_\pm(\xi,z)=\big(\lambda_\pm\big)^z\alpha_\pm(\xi),$$ 
so that
$$F^*\beta_\pm=\big(\lambda_\pm\big)^{z-1}\alpha_\pm\circ A=\big(\lambda_\pm\big)^z\alpha_\pm=\beta_\pm.$$
Hence $\beta_\pm$ is $F$-invariant and defines a 1-form on 
the quotient manifold $\cN=(\T^2\times\R)/\sim$. We use the same notation for these 1-forms. Then
  \begin{equation}\label{Edbeta}
    d\beta_\pm=\ln\lambda_\pm \, dz\wedge\beta_\pm.
  \end{equation}
  We consider the vector field $X=(0, 0, 1)$ on $\cN$. The lift of its flow to $\T^2\times \R$ is defined by $$\widetilde{\Phi_t}(\xi, z)=(\xi, z+t)$$
  hence the first return map to $\{z=0\}$ is $\Phi_1(\xi, 0)=(A\xi, 0)$ and is conjugate to $A$. The flow $(\Phi_t)$ is a suspension of $A$ and has the same Lyapunov exponents as $A$.
  
 We endow the manifold $\cM=\cN\times\R$ with the 1-form
 $$\Lambda=\beta_-+s\beta_+$$
 where $s$ is the $\R$-coordinate. We define 
 $\Omega=d\Lambda$. By \eqref{Edbeta}, we have $$\Omega=d\beta_-+ds\wedge \beta_++sd\beta_+=dz\wedge (\ln\lambda_-\beta_-+s\ln\lambda_+\beta_+)+ds\wedge\beta_+.$$
 Thus $\Omega^{\wedge 2}= 2\ln\lambda_- dz\wedge \beta_+\wedge ds \wedge \beta_+\not=0$ and $\Omega$ is a symplectic form.
 
 We define on $\cM$ the vector field $Y=X+2\ln \lambda_-\partial s$. Its flow is
 $$\psi_t(\xi, z, s)=(\Phi_t(\xi, z), \big(\lambda_-\big)^{2t}s).$$
 Hence $\cN\times \{ 0\}$ is the global attractor for $(\psi_t)$.
We have 
 \begin{align*}
 \psi_t^*\Omega=&dz\wedge\Big(\ln\lambda_-.\big(\lambda_-\big)^t\beta_-+\big(\big(\lambda_-\big)^{2t}.s\big)\ln\lambda_+.\big(\lambda_+\big)^t\beta_+\Big)+\\
 &\big(\lambda_-\big)^{2t}ds\wedge\big(\lambda_+\big)^t\beta_+.
 \end{align*}
 As $\lambda_-\lambda+=1$, we finally obtain
 $$\psi_t^*\Omega=\lambda_-^t\Omega.$$
\end{proof}

There are also examples of conformal symplectic diffeomorphisms on a non-exact symplectic manifold that have a non-isotropic invariant submanifold on which the restricted dynamics is Anosov.
\begin{proposition}\label{Pnonisotrop}
  There exists a conformal symplectic diffeomorphism $f$ on a
  6-dimensional symplectic manifold $({\cM}, \omega)$,with a
  4-dimensional invariant submanifold ${\cL}$ (hence ${\cL}$ is not
  isotropic).

  Moreover, the submanifold ${\cL}$ is the global attractor for $f$,
  $f_{|{\cL}}$ is conjugated to a hyperbolic automorphism of $\T^4$
  with $2$-dimensional stable and unstable foliations, and
  $f_{|{\cL}}$ is transitive with entropy equal to $- \log a$, where
  $a$ is the conformality ratio of $f$.
\end{proposition}

\begin{question*}
 In our example we have
    $a=\left(\frac{3+\sqrt{5}}{2}\right)^2$. In fact we can replace
    this number by the square of the largest eigenvalue of any Anosov
    automorphism of $\T^2$. We don't know if we can achieve other
    constants by a conformal symplectic diffeomorphisms  of the same symplectic manifold.
\end{question*}

\begin{proof}
  We consider the hyperbolic toral automorphism
  $T:\T^2\rightarrow \T^2$ that is defined by
  $T(\theta_1, \theta_2)=(2\theta_1+\theta_2, \theta_1+\theta_2)$. The
  associated linear map has eigenvalues $\lambda=\frac{3-\sqrt{5}}{2}<1$ and
  $\lambda^{-1}=\frac{3+\sqrt{5}}{2}>1$. Let $p = \frac{\sqrt{5}-1}{2}$. The unstable
  direction is spanned by $(1, p)$ and the stable
  one by $(1, -\frac{1}{p})$. The topological entropy is
  $-\log \lambda$ (see \cite{HaKa1995}).
 
  Then the product map
  $F=(T, T):(\theta_1, \theta_2, \theta_3, \theta_4)\in\T^2\times
  \T^2\mapsto (T(\theta_1, \theta_2), T(\theta_3, \theta_4))$ has   topological entropy equal to
  $-2\log \lambda$. We endow $\T^4$ with the
  closed 2-form $\Omega$ that is defined by
  $$\Omega=(d\theta_2 - p \, d \theta_1)\wedge
  (d\theta_4 - p \, d\theta_3).$$ 
  Observe that the kernel of $\Omega$ is the direction of the unstable
  foliation. Obviously,
  $F^*\Omega= \lambda^2\, \Omega$.  Now, we
  consider the subbundle
  $${\cM}=\left\{ (\theta, r)\in \T^4\times \R^4;
  r_2= p r_1\quad\text{and}\quad
  r_4=p  r_3\right\}$$ of $\T^4\times\R^4$.  This bundle
  corresponds to the tangent bundle to the unstable foliation in the
  identification of $T\T^4$ with $\T^4\times \R^4$.
  
  We denote by $\Omega_1$ the closed 2-form on ${\cM}$ that is equal to
  $\pi^*\Omega$ where $\pi: (\theta, r)\in{\cM}\mapsto \theta\in\T^4$
  and by $\Omega_2$ the restriction of the usual symplectic form
  $d\theta\wedge dr$ of $T^*\T^4$ to ${\cM}$:
  \begin{itemize}
  \item $\Omega_1=(d\theta_2-p \, d\theta_1)\wedge
    (d\theta_4-p\, d\theta_3);$  
  \item
    $\Omega_2=\frac{1}{5}(d\theta_2+ \frac{1}{p} d\theta_1)\wedge (dr_2+ \frac{1}{p}dr_1)+
    \frac{1}{5}(d\theta_4+ \frac{1}{p}d\theta_3)\wedge
    (dr_4+\frac{1}{p} dr_3).$
  \end{itemize}
  Let then $\omega=\Omega_1+\Omega_2$ be the chosen symplectic form on
  ${\cM}$.

  If we define $f~:{\cM}\rightarrow {\cM}$ by
  $f(\theta, r)=(T(\theta),\left(\frac{3-\sqrt{5}}{2}\right)^3r)$,
  then we have
\begin{itemize}
\item $f^*\Omega_1= \pi^*F^*\Omega =
  \lambda^2\Omega_1$;  
\item
  $f^*\Omega_2 = \frac{\lambda^3}{\lambda} \Omega_2 =
  \lambda^2\Omega_2$.
\end{itemize}
So finally $f:{\cM}\rightarrow {\cM}$ is a conformal symplectic
diffeomorphism such that
$f^*\omega=\lambda^2\omega$ and
$f^*(\T^4\times \{ 0\})=\T^4\times \{ 0\}$, where $\T^4\times \{ 0\}$
is not isotropic and the topological entropy of
$f_{|\T^4\times \{ 0\}}$ is $-2\log \lambda$.
\end{proof}

Given these counter-examples to isotropy, we start with the case of an
invariant surface (2-dimensional submanifold).

\begin{proposition}
  If a closed $C^1$ surface ${\cL}$ is invariant by a conformal { and non symplectic} $C^1$
  diffeomorphism of $({\cM}, \omega)$, ${\cL}$ is $\omega$-isotropic.

\end{proposition}

\begin{proof}
  We have $f^*\omega=a\, \omega$ for some $a\neq 1$.  Let ${\cL}$
  be an invariant 2-dimensional submanifold of $f$. We choose a
  finite atlas $\cA =\{ (U_i, \Phi_i) \}_{1\leq i\leq N}$ of
  ${\cL}$.  Endow ${\cL}$ with a Riemannian metric and define
  $$
  \begin{cases}
    \|\omega\|_{\cL,\infty}=\sup_{x\in{\cL}, u, v\in T_x{\cL}\setminus\{
    0\}}\frac{\vert\omega (u,v)\vert}{\| u\| \, \| v\|}\\
    \|D\Phi_i^{-1}\|_{\cL,\infty} = \sup_{u \in T\cL \setminus \{0\}} 
  \frac{\|D\Phi_i^{-1}(u)\|}{\|u\|}.
  \end{cases}$$  
  Then,
  $\int_U \omega$ is bounded over open subsets $U$ of $\cL$:
  \begin{equation}\label{E1}
    \left|\int_U\omega\right| \leq\sum_{i=1}^N\int_{\Phi_i(U_i)} \|\omega\|_{\cL,\infty} \, \|
    D\Phi_i^{-1}\|_{\cL,\infty}^2.d\text{Leb}
  \end{equation} 
  Now, let $U$ be an open set of $\cL$ and $n \in
  \Z$. Since $f^nU$ is an open subset of $\cL$ and 
  $$\int_{f^nU}\omega=a^n\int_U\omega,$$
  we see that $\int_U \omega$ must be zero. Thus $\cL$ is isotropic.
\end{proof}

If $\cL$ has any dimension, the same conclusion holds provided some
constraint on the topological entropy $\ent (f_{|\cL})$ of the
dynamics carried by $\cL$.  Define the spectral
radius of a self-map $g$ as 
\[\rad (Dg) = \limsup_{n\rightarrow+\infty} \| Dg^{n} \|_\infty^{\frac{1}{n}}.\]

\begin{proposition}\label{PYomdin}
  Let $f$ be a conformal diffeomorphism of a symplectic manifold
  $(\cM, \omega)$, i.e. such that $ f^*\omega=a\,\omega$ with
  $a\in ]0, 1[$.  Let ${\cL}$ be an 
  invariant
  closed submanifold. Assume one of the following hypothesis.
  \begin{enumerate} 
  \item The diffeomorphism $f$ is smooth{, $\cL$ is
      smooth} and
    \[\ent(f_{|{\cL}}) < - \log(a);\]
  \item The diffeomorphism $f$ and $\cL$ are $C^r$ 
    for some $r\geq 1$ and
    \[\ent(f_{|{\cL}}) + \log^+\left(\text{\rm Rad}
        (Df^{-1}_{|{\cL}})^{2/r}\right) < -\log(a).\]
  \end{enumerate}
  Then ${\cL}$ is $\omega$-isotropic.
\end{proposition}

\begin{proof}
   We assume that $\cL$ is invariant and not
    isotropic. There exists a constant $k>0$ such that on $\cL$, we
    have $\vert\omega\vert\leq k\vert \mathrm{vol}\vert$ where
    $\mathrm{vol}$ is the 2-dimensional volume form induced by the
    Riemannian metric. We
  choose in $\cL$ a small piece $\cS$ of symplectic surface (whose tangent space
  intersects the characteristic bundle of $\cL$ only in $0$). Then
  $\omega(f^{-n}(\cS))=a^{-n}\omega(\cS)\not=0$ and then
  \[\limsup_{n\rightarrow \infty} \frac{1}{n}\log\Big\vert \mathrm{vol}(f^{-n}
  (\cS)\Big\vert\geq  \lim_{n\rightarrow\infty}\frac{1}{n}\big(\log\vert
  \omega(f^{-n}(\cS))\vert-\log k\big)=-\log(a).\] 
  The conclusion follows from Yomdin's inequality, which we have recalled in appendix~\ref{AYomdin}.
\end{proof}

{ \begin{remark}
  This statement implies in particular that if $\cL$ is an invariant submanifold by a conformal flow $(\varphi_t)$ then
  \begin{itemize}
      \item if $\cL$ and $(\varphi_t)$ are $C^1$ and if $\varphi_{t|\cL}$ is $C^1$ conjugate to a rotation on a torus for some $t\not=0$, then $\cL$ is isotropic; indeed, in this case, the entropy vanishes and the spectral radius of $Df$ is 1. A simpler proof of this statement is given in \cite{CaCeLlA2013}.
       \item if $\cL$ and $(\varphi_t)$ are smooth and if $\varphi_{t|\cL}$ is $C^0$ conjugate to a rotation on a torus for some $t\not=0$, then $\cL$ is isotropic; indeed, in this case, the entropy vanishes. 
  \end{itemize}
\end{remark}}

\section{Entropy}\label{Sentropy}

  The purpose of this section is to improve regularity in  Proposition~\ref{PYomdin}. We will start by giving an abstract result on a manifold endowed with a form with constant rank and then we will give an application to invariant submanifolds of conformal symplectic dynamics.


Let
\begin{itemize}
\item $\cN^{(n)}$ be a compact Riemannian $C^2$ manifold and $d$ its distance
\item $\cF$ be a $C^2$ foliation induced by a subbundle $F$ of $T\cN$ of rank $p{\leq n-1}$
\item $\Omega$ be an $(n-p)$-form on $\cN$ which induces a volume on
  submanifolds transverse to $\cF$
\item $f$ be a $C^1$-diffeomorphism of $\cN$ preserving $\cF$ and such
  that
  \[f^*\Omega = b \, \Omega\]
  for some $b>1$.
\end{itemize}

\begin{theorem}
  \label{thm:ent}%
  The topological entropy of $f$ satisfies
  \[\ent f \geq \ln b.\]
\end{theorem}

\begin{proof}
  Key to the proof is the refined distance $d_\cF$ on $\cN$ defined by
  \[d_\cF(x,y) =
  \begin{cases}
    \infty \quad \mbox{if $x$ and $y$ are not on the same leaf}\\
    \mbox{distance from $x$ to $y$ along their common leaf otherwise}.
  \end{cases}
  \]

  \begin{lemma}\label{Ldistleaf}
There exist $\varepsilon>0$ and  $K>0$ such that for every $x,y \in \cN$ 
  \begin{equation}
    \label{eq:dF}
    d_\cF(x,y)<\varepsilon \quad \Rightarrow d_\cF(x,y) \leq K d(x,y).
  \end{equation}
  \end{lemma}
Replacing the Riemannian metric $d$ by $\frac{1}{\varepsilon}d$, we will asssume that $\varepsilon=1$. 
\begin{proof}[Proof of Lemma \ref{Ldistleaf}]
We choose $\varepsilon>0$ that is strictly less than the radius of injectivity of the metric $d$ restricted to every leaf and introduce
$$\cD=\{ (x, y)\in\cN\times\cN; d_{\cF}(x, y)\leq\varepsilon\}.$$
This set is closed and due to our choice of $\varepsilon$, $d_\cF$ is continuous on $\cD$. If we use the notation 
$$\Delta=\{ (x, x); x\in\cN\},$$
then the continuous function  $\frac{d_\cF}{d}$ is bounded on the complement of every neighbourhood of $\Delta$ in $\cD$.

The exponential maps for the Riemannian form $g$ and for the Riemannian form $g_\cF$ restricted to the leaves are tangent along the tangent bundle to the leaves, hence
$$\lim_{(x, y)\rightarrow \Delta}\frac{d_\cF(x, y)}{d(x,y)}=1.$$
\end{proof}
  
  For every $x \in \cN$, let $\cU_x^{(n-p)}$ be a submanifold through
  $x$ of dimension $n-p$, transverse to $\cF$. Let $\cV_x$ be a
  tubular neighborhood of $\cU_x$, of the form
  \[\cV_x = \cup_{y\in \cU_x} \{z\in \cN, \; d_\cF(y,z)<\e_x\}.\]
  We choose $\cU_x$ and $\e_x<1$ small enough so that $\cV_x$ has
  a product structure. Furthermore, let 
  \[\cW_x = \cup_{y\in \cU_x} \{z\in \cN, \; d_\cF(y,z)<\e_x/2\}.\]
  Let $\cF_{\cW_x}$ be the foliation induced on $\cW_x$ by $\cF$. (Due
  to the product structure, leaves of $\cF_{\cW_x}$ are of the form
  $\cW_x \cap L_x$, where $L_x$ is the leaf through $x$ of the
  foliation induced on $\cV_x$.) The neighborhood $\cW_x$ has the
  property that for any two points $y$ and $z$ of $\cW_x$, if
  $d_\cF(y,z) <\e_x/2$ then $y$ and $z$ must belong to the same leaf
  of $\cF_{\cW_x}$; indeed, if $y$ and $z$ do not lie on the same leaf
  of $\cF_{\cW_x}$, their distance must be $\geq\e_x$ since any path
  from $y$ to $z$ along a leaf of $\cF$ runs twice across
  $\cV_x \setminus \cW_x$.

  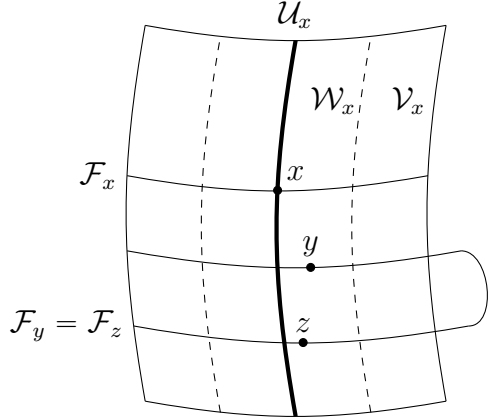
\begin{figure}[!htbp]
    \centering
    \begin{tikzpicture}[scale=1]
        \draw (0,0) to[out=100,in=260] (0,5);
        \draw (0,5) to[out=-10,in=190] (4,5);
        \draw (4,0) to[out=100,in=260] (4,5);
        \draw (0,0) to[out=-10,in=190] (4,0);

        \draw[dashed] (1,-0.15) to[out=100,in=260] (1,4.85);
        \draw[dashed] (3,-0.15) to[out=100,in=260] (3,4.85);

        \draw[ultra thick] (2,-0.2) to[out=100,in=260] (2,4.8) node[anchor=south] {$\cU_x$};
        \draw (-0.24,3) node[anchor=east] {$\cF_x$} to[out=-10,in=190] (3.76,3);

        \draw (-0.15,1) node[anchor=east] {$\cF_y=\cF_z$} to[out=-10,in=190] (4.3,1);
        \draw (-0.24,2) to[out=-10,in=190] (4.2,2);
        \draw (4.3,1) to[out=0,in=0] (4.2,2);

        \filldraw (1.76,2.805) circle (1.5pt);
        \node[align=right,above] at (2,2.8) {$x$};

        \filldraw (2.2,1.78) circle (1.5pt) node[above] {$y$};
        \filldraw (2.1,0.78) circle (1.5pt) node[above] {$z$};

        \node at (2.5,4) {$\mathcal{W}_x$};
        \node at (3.5,4) {$\mathcal{V}_x$};
    \end{tikzpicture}
    \caption{Construction of the finite covering of $\cN$}
  \end{figure}
  
Let $\cW_{x_1}$, ..., $\cW_{x_I}$ be a finite subcovering of $\cN$. Denote $\cW_{x_i}$ by $\cW_i$, and let $\e = \min_i \e_{x_i}/2$. So, the following property holds:

  \begin{itemize}
  \item[(*)] For every $i =1,...,I$ and $y,z \in \cW_i$ such that
    $d_\cF(y,z)< \e$, $y$ and $z$ belong to the same leaf of the
    foliation $\cF_{\cW_i}$ induced by $\cF$ on $\cW_i$.
  \end{itemize}
  
  Moreover, since $f^{-1}$ and $\cF$ are continuous and $f$ preserves
  $\cF$, there exists $\eta < \e$ such that

  \begin{itemize}
  \item[(**)] For every $x,y \in \cN$ such that $d_\cF(x,y) < \eta$,
    $d_\cF(f^{-1}x, f^{-1}y) < \e$.
  \end{itemize}
  
  According to Lebesgue   covering lemma, there exists
  $\theta < \eta/K$ such that every ball of radius $\theta$ is
  inside at least one of the $\cW_{i}$'s.

  Let $(Q_j)_{1 \leq j \leq J}$ be a decomposition of $\cN$ into cubes
  (or compact submanifolds with boundaries) such that each cube is
  contained in a ball of radius $< \theta$.

  Let $\cS$ be a submanifold of $\cN$ of dimension $n-p$, included into some cube $Q_j$ and transverse to $\cF$. $\cS$ must lie into some $\cW_{i}$. For any $\cW_i$ containing $\cS$, $\cS$ meets each leaf of $\cF_{\cW_i}$ at isolated points. By narrowing $\cS$, we may assume that $\cS$ meets each leaf of $\cF_{\cW_{i}}$ at one point at most.

  We claim that

  \begin{itemize}    
  \item[(***)] For every $k$ and $j_1,...,j_k \in \{1,...,J\}$, 
    \[f^{k}(\cS) \cap f^{k-1}(Q_{j_1}) \cap \cdots \cap Q_{j_k}\]
    meets each leaf of any $\cW_i$ containing $Q_{j_k}$ at one point at most.
  \end{itemize}
  
  Let $j \in \{1,...,J\}$. Then $\cS' = f(\cS) \cap Q_j$
  is also transverse to the foliation. Let $x,y \in \cS'$ be on a
  common leaf of $\cF_{\cW_{i}}$, with $Q_j \subset \cW_{i_0}$. Since such
  leaves have a diameter $<1$ (due to our choice $\e_x<1$),
  using~\eqref{eq:dF}\footnote{Recall the metric was changed in order to have $\varepsilon=1$ in \eqref{eq:dF}.}, we see that
  \[d_\cF(x,y) \leq K d(x,y) \leq K \mathrm{diam}\, Q_j \leq K \theta
  \leq \eta.\]
  Using (**), $d_\cF(f^{-1}x,f^{-1}y) < \e$. But using (*), $f^{-1}x$
  and $f^{-1}y$ belong to the same leaf of $\cF_{\cW_{i_0}}$. So, by the
  constructing property of $\cW_{i_0}$, $f^{-1}x=f^{-1}y$ and $x=y$. By
  induction, (***) holds.

  If $\cS \subset \cW_i$,  we have
  \[\left|\Omega\left(f^k(\cS) \cap f^{k-1}(Q_{j_1}) \cap \cdots \cap Q_{j_k} \right) \right| \leq \max \{|\Omega(\cU_1)|, ...,
  |\Omega(\cU_I)|\} = M,\] uniformly with respect to $k$. Let
  \[N_k = \sharp \large\{(j_1,...,j_k), \; f^k(\cS) \cap
  f^{k-1}(Q_{j_1}) \cap ... \cap Q_{j_k} \neq \emptyset\large\}.\]
  Then
  \[b^k\, |\Omega(\cS)| \leq N_k M,\]
  hence
  \[\frac{1}{k} \ln N_k \geq \frac{1}{k} \ln
  \frac{|\Omega(\cS)|}{M} + \ln b,\]
  hence the wanted inequality.
\end{proof}

Now assume that $\omega$ is a presymplectic form\footnote{A
  presymplectic form is a a closed $2$-form with constant rank.} of $\cN$ of (even) rank $2\ell{ \geq 2}$ and
\[f^*\omega = a \, \omega, \quad a >1.\]
The kernel of $\omega$ is a uniquely integrable subbundle $F$ of corank $2\ell$. Setting $\Omega=\omega^\ell$ and $b=a^\ell$
brings us back to the prior setting.

\begin{corollary}
  \label{cor:ent}%
  The topological entropy of $f : \cN \righttoleftarrow$ satisfies
  \[\ent f \geq \frac{\rank (\omega)}{2} \ln a.\]
\end{corollary}

Let us now return to our usual setting, where  $(\cM,\omega)$ is a symplectic manifold.

\begin{corollary}\label{CCisoento}
    Let $f:\cM\righttoleftarrow$ be a $C^3$ conformal symplectic diffeomorphism such that $f^*\omega=a\omega$ with $a>1$. Suppose that $\cN$ is an invariant $C^3$ submanifold such that the induced form $\omega_{\vert \cN}$ on $\cN$ has constant rank. Then
    \[\ent f_{\vert\cN} \geq \frac{\rank (\omega_{\vert \cN})}{2} \ln a;\]
    in particular, if the entropy of $f_{|\cN}$ vanishes, $\cN$ is isotropic. 
    
\end{corollary}

Note that if $\cN$ is a  compact submanifold such that $f_{|\cN}$ is minimal,\footnote{By definition, it is minimal if every orbit is dense} $\omega_{|\cN}$ has constant rank and so the corollary applies.  

\begin{proof}
As $\cN$ is $C^2$, its tangent bundle is $C^1$. Then Frobenius Theorem applies to $F = \ker \omega_{\vert \cN}$\footnote{The infinitesimal integrability condition is well known: if $X,Y$ are sections of $F$ and $Z$ is a section of $T\cN$, $0 = d\omega(X,Y,Z) = - \omega([X,Y],Z)$, which shows that $[X,Y]$ itself is a section of $F$.} and the characteristic foliation $\cF$ exists.
\end{proof}

\section{Liouville class of invariant submanifolds}\label{SLiouville}

\label{sec:exactnessdiff}
In this section we assume that $(\cM,\omega={ -d\lambda})$ is an exact symplectic manifold. The goal is to prove that, given a conformal dynamics,  there is only one Liouville class that  an isotropic invariant submanifold may have.

\subsection{Action of conformal dynamics on Liouville classes}

\begin{definition}
    Let $j : \cS \hookrightarrow \cM$ be an isotropic embedding.
    
    $\bullet$ Its \emph{Liouville class} $[j] \in H^1(\cS,\R)$  is  the cohomology class of the induced form $j^*\lambda$. 
    
    $\bullet$ It is \emph{exact} if its Liouville class vanishes.
\end{definition}

So, except if the submanifold is exact, its Liouville class depends on the chosen embedding with a given image.

When $\cM=T^*\cQ$ is the cotangent bundle of a closed manifold endowed with its tautological 1-form $\lambda$ { and $\cL$ is a Lagrangian submanifold of $T^*\cQ$ that is homotopic to the zero section $\cZ$, the restriction  to $\cL$ of the canonical projection $\pi:T^*\cQ\rightarrow \cQ$ is a homotopy equivalence between $\cL$ and $\cQ$ and  induces an isomorphism between $H^1(\cL,\R)$ and $H^1(\cQ,\R)$. Denoting by $j_\cL:\cL\hookrightarrow T^*\cQ$ the canonical injection defined by $j_\cL(x)=x$, the Liouville class of the submanifold $\cL$ is the cohomological class
$$[\cL]=\Big[\big(\pi_{|\cL}\big)_*\big(j_\cL^*\lambda\big)\Big]\in H^1(\cQ,\R).
$$
}
In this case, we may thus update the definition of Liouville classes.

\begin{definition}
 Let  $\cL$ be a Lagrangian submanifold of $T^*\cQ$ that is 
 { homotopic to the zero section,} the Liouville class $[\cL]$ of $\cL$ is the cohomology class on $\cQ$ whose pull back  by $\pi_{|\cL}$ is the cohomology class of $\lambda_{|T\cL}$. 
\end{definition}

The following straightforward proposition explains that the group of conformal dynamics acts {on the set of} Liouville classes of isotropic embeddings { that are homotopic to a given isotropic embedding 
} of a given manifold $\cS$ by homotheties (translations when the dynamics is symplectic).

\begin{proposition}\label{Phomothety}
  Let $f: \cM \righttoleftarrow$ be a conformal diffeomorphism with conformality ratio $a$. Then $\eta=f^*\lambda-a\lambda$ is a closed 1-form.
  
 { Let $  j_0:\cS\hookrightarrow\cM$ be an isotropic embedding.  For every isotropic embedding $ j:\cS\hookrightarrow\cM$ that is homotopic to $j_0$, the Liouville class of the isotropic embedding $f\circ j:\cS\hookrightarrow\cM$  is
  $$[f\circ j]=a[j]+[j_0^*\eta].$$}
\end{proposition}
  
\begin{definition}
  A diffeomorphism $f: \cM \righttoleftarrow$ is $\lambda$ conformal Hamiltonian (CH) if  there exists an isotopy $(f_t)_{t\in[0, 1]}$ such that $f_0={\rm Id}_\cM$, $f_1=f$ and two functions $H:[0, 1]\times \cM\rightarrow \R$ and $\alpha:[0, 1]\rightarrow \R$ such that 
  $$\forall (t, x)\in [0, 1]\times \cM, i_{\dot f_t(x)}\omega=\alpha(t)\lambda+\partial_xH(t, x).$$
\end{definition}

\begin{remark} A diffeomorphism $f: \cM \righttoleftarrow$ is conformal Hamiltonian if and only if  there exists an isotopy $(f_t)_{t\in[0, 1]}$ of CES diffeomorphisms such that $f_0={\rm Id}_\cM$ and $f_1=f$.
  \end{remark}
  
  \begin{definition}
The flow $(\varphi_t)$ associated to the vector field $X$ on $\cM$ is  $\lambda$ conformal Hamiltonian  if there exists $\alpha\in \R$ and $H:\cM\rightarrow \R$ such that $i_X\omega=\alpha \lambda +dH$.
\end{definition}
\begin{remark}
  A flow is a flow of  $\lambda$ conformal exact symplectic diffeomorphisms if and only if it is $\lambda$ conformal Hamiltonian.
\end{remark}

To describe the behavior of Lagrangian submanifolds of $T^*\cQ$ that are H-isotopic to a graph, we first need the following invariance result. 
{
\begin{proposition}\label{Pisotopytograph} Let $(\cL_t)$ be an isotopy of Lagrangian submanifolds of $T^*\cQ$  such that $\cL_0=\cZ$. Then $\cL_1$ is  H-isotopic to a graph. \\
  \end{proposition}
 }

\begin{corollary}

\label{PHisotopytograph}
  Let $(g_t)_{t\in[0, 1]}$ be an isotopy of conformal symplectic diffeomorphisms such that $g_0={\rm Id}_{T^*\cQ}$. Let $\cL$ be a Lagrangian submanifold of $T^*\cQ$ that is H-isotopic to a graph. Then $g_1(\cL)$ is H-isotopic to a graph. \\
  If moreover $\cL$ is H-isotopic to the zero-section and the isotopy is conformal Hamiltonian, then $g_1(\cL)$ is H-isotopic to the zero-section.
\end{corollary}
 {  \begin{proof}[Proof of Proposition \ref{Pisotopytograph}] 
   We will prove
  \begin{lemma}\label{Lisotexactlagr} Assume that $\cL$ is H-isotopic to the zero section  and that $(\cL_t)_{t\in[-\varepsilon, \varepsilon]}$ is an isotopy of exact Lagrangian submanifolds  such that $\cL_0=\cL$. Then  there exists a neighbourhood $\cN$ of $0$ in $[-\varepsilon, \varepsilon]$ such that for every $t\in\cN$, $\cL_t$ is H-isotopic to the zero section.
  \end{lemma}
  \begin{proof}[Proof of Lemma \ref{Lisotexactlagr}] 
   We use Weinstein tubular neighbourhood Theorem, \cite{Weinstein1977}. Let $\cT$ be a symplectic tubular of $\cL$, i.e. there exists a neighbourhood $\cU$ of the zero section in $T^*\cL$ and a symplectic embedding $\phi:\cU \hookrightarrow T^*\cQ$ with image $\cT$ that is ${\rm Id}_{\cL}$ on $\cL$. As $\Phi$ maps the exact Lagrangian submanifold $\cL$ of $T^*\cL$ onto the exact Lagrangian submanifold $\cL$ of $T^*\cQ$, then $\Phi$ is exact symplectic. \\
 This implies that every submanifold $\phi^{-1}(\cL_t)$ is exact Lagrangian. Moreover, there exists a neighbourhood $\cN$ of $0$ in $[-\varepsilon, \varepsilon]$ such that for every $t\in\cN$, $\phi^{-1}(\cL_t)$ is a graph. Hence this is the graph of an exact 1-form $du_t$. 
  
 Then $\phi^{-1}(\cL_t)$ is the image by the time-1 Hamiltonian flow of $H=-\frac{du_t}{dt}\circ \pi$. Using a bump function, we can assume that $H$ has support in $\cU$, and then the time-1 map of the Hamiltonian $H\circ\phi$ maps $\cL$ onto $\cL_t$.
  
  \end{proof}
We now prove Proposition \ref{Pisotopytograph}.  Let us firstly deal with the case when all the $\cL_t$s are exact. We introduce  
  $$\{ t\in[0, 1]; \forall s\in [0, t], g_s(\cL)\;{\rm is}\; H-{\rm isotopic}\;{\rm to}\;{\rm the}\;{\rm zero}\;{\rm section}\}.$$
  Lemma \ref{Lisotexactlagr}  and the transitivity of the relation of $H$-isotopy imply that this set is closed and open in $[0, 1]$, hence equal to $[0,1]$.
  
 Now we just assume that $(\cL_t)$ is an isotopy of Lagrangian submanifolds of $T^*\cQ$  such that $\cL_0=\cZ$. We choose an arc $(\eta_t)_{t\in[0,1]}$ of closed 1-forms on $\cQ$ whose cohomology class $[\eta_t]=[\cL_t]$ is the Liouville class of $\cL_t$. We denote by $T_t:T^*\cQ\righttoleftarrow$ the symplectic diffeomorphisms such that $T_t(p)=p+\eta_t\circ\pi (p)$. Then $\cL^*_t=T_{-t}(\cL_t)$ defines a homotopy of exact Lagrangian submanifolds of $T^*\cQ$. A result of the first part of the proof is that $\cL^*_t$ is $H$-isotopic to the zero section, i.e. there exists a $H$-isotopy $(\phi_t)_{t\in[0, 1]}$ such that $\phi_0={\rm Id}$ and $\phi_1(\cZ)=\cL^*_1$. Hence $\cL_1=T_1(\cL^*_1)$ is $H$ isotopic to the graph of $\eta_1$ via  the $H$-isotopy 
 $$(\psi_t)_{t\in[0,1]}=(T_1\circ \phi_t\circ T_1^{-1})_{t\in[0, 1]}.$$

  \end{proof}
  
  \begin{proof}[Proof of Corollary  \ref{PHisotopytograph}]
  We assume that  $(g_t)_{t\in[0, 1]}$ is an isotopy of conformal symplectic diffeomorphisms such that $g_0={\rm Id}_{T^*\cQ}$ and that  $\cL$ is a Lagrangian submanifold of $T^*\cQ$ that is H-isotopic to a graph. Then  there exist a closed $1$-form $\eta$ on $\cQ$ and  a $H$-isotopy $(h_t)_{t\in[0,1]}$ such that $h_0={\rm Id}_{T^*\cQ}$ and  $\cL=h_1({\rm graph}(\eta))$. We introduce the symplectic diffeomorphisms $(T_t)_{t\in[0, 1]}$ of $T^*\cQ$ that are defined by $T_t(p)=p+t\eta\circ \pi(q)$. Then $$(\cL_t)_{t\in[0,1]}=(g_t\circ h_t\circ T_t(\cZ))_{t\in[0,,1]}$$ is a isotopy of Lagrangian submanifolds such that $\cL_0=\cZ$ and $\cL_1=g_1(\cL)$. A result of Proposition \ref{Pisotopytograph} is that $g_1(\cL)$ is H-isotopic to a graph. \\
  
  If moreover $\cL$ is H-isotopic to the zero-section and the isotopy is conformal Hamiltonian, then all the maps $g_t\circ h_t\circ T_t$ are conformal Hamiltonian  and thus every manifold $\cL_t$ is exact Lagrangian. The conclusion is a result of the second part of Proposition \ref{Pisotopytograph}.

  \end{proof}}
  
  
  
\subsection{Liouville classes of invariant  submanifolds}

Let $  j_0:\cS\hookrightarrow\cM$ be an isotropic embedding. We denote by $\cJ(j_0)$ the set of isotropic embeddings $  j:\cS\hookrightarrow\cM$ that are homotopic to $j_0$. 

{ A consequence of Proposition \ref{Phomothety} is
\begin{proposition}\label{Puniqueclass}
Let $f: \cM \righttoleftarrow$ be a conformal diffeomorphism.
Let $j\in\cJ(j_0)$ be an isotropic embedding which is strongly $f$-invariant in the sense that\\
$\bullet$ $j(\cS)=f\circ j(\cS)$\\
$\bullet$ $f$ acts trivially on $H^1(j(\cS),\R)$.\\
Then $j$ may have only one Liouville class, that we denote by $[\ell_f(\cJ(j_0))]$.\\
In particular, when $f$ is CES, then $[\ell_f(\cJ(j_0))]=0$ and $j$ has to be exact.
\end{proposition}}

{ \begin{proof}
Let $  j:\cS\hookrightarrow\cM$ be such an embedding. 
With the notations of Proposition \ref{Phomothety}, we have
$$[f\circ j]=a[j]+[j_0^*\eta].$$
As $f$ acts trivially on on $H^1(j(\cS),\R)$, we have $[f\circ j]=[j]$ and finally $[j]$ has to be  the only fixed point of the homothety that maps $[j]$ on $a[j]+[j_0^*\eta]$. 
\end{proof}}

As a consequence:

\begin{proposition} \label{Pinvalagrexact}
 Let $f:\cM \righttoleftarrow$ be a $\lambda$ CES diffeomorphism. 
 Then every invariant  isotropic submanifold { $\cS$  such that $f_{|\cS}$ acts trivially on $H^1(\cS)$} is exact. 
\end{proposition}

{}

 
 \begin{corollary}\label{Cuniqueclassflow}
Let $X$  be a CS vector field  on $\cM$ with flow $(\varphi_t)$. Let $  j_0:\cS\hookrightarrow\cM$ be an isotropic embedding. We denote by $\cJ(j_0)$ the set of isotropic embeddings $  j:\cS\hookrightarrow\cM$ that are homotopic to $j_0$. 
Then there is only one Liouville class that we denote by $[\ell_X(\cJ)]$, that an isotropic embedding $j\in\cJ(j_0)$ such that
$$\forall t\in\R, \varphi_t(j(\cS))=j(\cS)$$
may have.\\
In particular, when $X$ is CH, then $[\ell_X(\cJ)]=0$.
\end{corollary}

 \begin{corollary}\label{Cuniqueclasscotangent}
Let $f: T^*\cQ \righttoleftarrow$ be a CS-diffeomorphism that is homotopic to ${\rm Id}_{T^*\cQ}$. Then there is only one Liouville class that we denote by $[\ell_f]$, that a  homotopic  to the zero section
 and $f$-invariant submanifold may have.
\end{corollary}
\begin{proof} Let $j_0: \cZ{}\hookrightarrow T^*\cQ$ be the canonical injection of the zero-section.\\ 
We assume that $\cL$ is   an $f$-invariant submanifold that is H-isotopic to a Lagrangian graph. Let $  j:\cL\hookrightarrow T^*\cQ$ be the canonical injection. With the notations of Proposition \ref{Puniqueclass}, we have $j\in \cJ(j_0)$.\\
Because\\
$\bullet$  $\pi_{|\cL}$ defines an homotopy equivalence between $\cL$ and $\cQ$;\\
$\bullet$ $\pi$ defines an homotopy equivalence between $T^*\cQ$ and $\cQ$;\\
$\bullet$ $f$ is homotopic to ${\rm Id}_{T^*\cQ}$,\\
then $f$ acts trivially on $H^1(\cL,\R)$.\\
A result of Proposition \ref{Puniqueclass} is that $[j]=[\ell_f(\cJ(j_0))]$, i.e. $[\cL]=[\ell_f(\cJ(j_0))]$.

\end{proof}

\section{Liouville class of Lagrangian submanifolds of \texorpdfstring{$T^*\cQ$}{T*Q} with compact orbits}

The goal of this section is to prove that, given a conformal dynamics on $T^*\cQ$, there is only one Liouville class that a   Lagrangian submanifold with compact orbit may have.

We assume that $\cM=T^*\cQ$ and that $f: \cM \righttoleftarrow$   is CS-isotopic to ${\rm Id}_\cM$. 


We suppose that $j:\cQ\hookrightarrow\cM$ is a Lagrangian embedding such that $j(\cQ)=\cL$  is H-isotopic to a graph and has compact orbit (for example is contained in some compact attracting set). 

\begin{theorem}\label{Tcompactexactdiff} 
If $f: \cM \righttoleftarrow$   is a $\lambda$  CES diffeomorphism  that is CS-isotopic to ${\rm Id}_\cM$ and $\cL$ is a 
{  Lagrangian submanifold that is isotopic to the zero section among the Lagrangian submanifolds} of $T^*\cQ$ such that $\bigcup_{k\in\Z}f^k(\cL)$ is relatively compact, then $\cL$ is exact.
\end{theorem}

\begin{corollary}\label{Ccompactdiff}
Let  $f: \cM \righttoleftarrow$   be a   diffeomorphism that is CS-isotopic to ${\rm Id}_\cM$ and let  $\cL$ be a {  Lagrangian submanifold that is isotopic to the zero section among the Lagrangian submanifolds} such that $\bigcup_{k\in\Z}f^k(\cL)$ is relatively compact, then $[\cL]=\ell_f$.
\end{corollary}

\begin{corollary}
Let  $(\varphi_t)$ be the flow of the    conformal symplectic vector field $X$ and let  $\cL$ be a {  Lagrangian submanifold that is isotopic to the zero section among the Lagrangian submanifolds} of $T^*\cQ$ such that $\bigcup_{t\in\R}\varphi_t(\cL)$ is relatively compact, then $[\cL]=\ell_X$.
\end{corollary} 
\begin{remark} We give a proof of Theorem \ref{Tcompactexactdiff} that uses the notion of graph selector. If $\cQ$ (as $\T^n$) satisfies that every element of $H^1(\cQ, \R)\backslash\{ 0\}$  contains a non-vanishing $1$-form, we can give a simpler proof.
  Indeed, in the proof, we are reduced to prove that if we have a sequence $(\cL_n)$ of Lagrangian submanifolds such that $[ \cL_n]=k^n[\cL_0]$ tends to infinity as $n \to  \infty$, then $\bigcup_{n\in\N}\cL_n$ is not relatively compact. If the 1-form $\eta$ on $\cQ$ represents
  $[\cL_0]$, then $\cL_n$ and the graph of $k^n\eta$ intersect. As $\eta$ doesn't vanish, we can conclude.   \end{remark}

  \begin{proof}[Proof of Theorem \ref{Tcompactexactdiff}] We endow $\cQ$ with a Riemannian metric and denote by $\|.\|$ the norm on $T\cQ$. \\
  As $f$ is CES and $f^*\lambda-a\lambda$ is exact, $f^k$ is also CES with
  $$(f^k)^*\lambda-a^k\lambda=\sum_{j=0}^{k-1}a^{k-j-1}(f^j)^*(f^*\lambda-a\lambda)$$
  is exact.\\
Suppose ad absurdum that $[\cL]$ is not $0$.  Let $\eta$ be a $1$-form on $\cQ$ representing
  $[\cL]$. There is a loop $\gamma:\T\rightarrow\cQ$ such that $\int_\gamma \eta\not=0$.

As $f$ is CS-isotopic to  ${\rm Id}_\cM$ and by transitivity of the relation of CS-isotopy, $f^k$ is also CS-isotopic to  ${\rm Id}_\cM$. Hence by Corollary \ref{PHisotopytograph}, $f^k(\cL)$ is H-isotopic to a graph. The submanifold $\cL$ is H-isotopic to the graph of
  $\eta$.   A result of Proposition \ref{Phomothety} is that 
  $f^k(\cL)$ is H-isotopic to the graph of
  $a^k\eta$. If we denote by $(\tau_t):\cM\righttoleftarrow$ the flow of symplectic diffeomorphisms $\tau_t(p)=p+t\eta(\pi(p))$, then $\tau_{-a^k}\circ f^k(\cL)$ is H-isotopic to the zero section and then admits a generating function and a graph selector that is (see e.g. \cite{Sib2004} p 98 and references herein) a Lipschitz function $u_k:\cQ\rightarrow \R$ that is $C^1$ on an open subset $\cU_0$ of $\cQ$ with full Lebesgue measure such that 
  \[ \forall q\in\cU_0, du_k(q)\in \tau_{-a^k}\circ f^k(\cL).\]
  Using Fubini theorem, we find a loop $\gamma_k$ that is $C^1$ close to $\gamma$ and such that
  \begin{itemize}
  \item $\gamma_k$ is smooth and isotopic to $\gamma$;
  \item for Lebesgue almost $s\in\T$, we have $\gamma_k(s)\in\cU_0$.
  \end{itemize}
  As $u_k\circ \gamma_k $ is Lipschitz and then absolutely continuous, we have
  $$0=\int_\T \frac{d(u_k\circ \gamma_k)}{ds}(s)ds.$$
  Because $\gamma_k(s)\in\cU_0$ for almost every $s$, we deduce
  $$0=\int_\T du_k(\gamma_k(s)).  \gamma_k'(s)ds$$
  and because $\gamma_k$ is homotopic to $\gamma$ and $\eta$ is closed,
  \[\int_\T \Big( a^k\eta(\gamma_k(s))+du_k(\gamma_k(s))\Big).  \gamma_k'(s)ds=a^k\int_\gamma\eta
  \]
  As the loops $\gamma_k$ are $C^1$-close to $\gamma$, there exists a constant $K$ that is a upper bound for all the $\| \gamma_k'(s)\|$. Hence there is a subset $E_k$ with non-zero Lebesgue measure of $\T$ such that for every $s\in E_k$, we have  
\begin{equation}\label{Egrande}\|a^k\eta(\gamma_k(s))+du_k(\gamma_k(s)\|\geq \frac{a^k}{2K}\Big\vert\int_\gamma\eta\Big\vert.\end{equation}

  Moreover, for almost every $s\in\T$, we have 
  \[du_k(\gamma_k(s))\in \tau_{-a^k}\circ f^k(\cL)\]
i.e.
\begin{equation}\label{Eselec} a^k\eta(\gamma_k(s))+ du_k(\gamma_k(s))\in  f^k(\cL).\end{equation}
  We deduce from \eqref{Egrande} and \eqref{Eselec} that there is $p\in f^k(\cL)$ such that $\| p\|\geq \frac{a^k}{2K}$.

\end{proof} 
\begin{proof}[Proof of Corollary \ref{Ccompactdiff}] Let $(f_t)_{t\in [0, 1]}$ be an isotopy of conformal symplectic diffeomorphisms such that $f_0={\rm Id}_\cM$ and $f_1=f$.    By Proposition \ref{PCES}, see Appendix \ref{Aconfexact}, we know that there is a  diffeomorphism $g: \cM \righttoleftarrow$ symplectically isotopic to ${\rm Id}_\cM$ such that   $g\circ f\circ g^{-1}$ is $\lambda$ CES. Then $(h_t)=(g\circ f_t\circ g^{-1})_{t\in [0, 1]}$ is an isotopy of conformal symplectic diffeomorphisms such that $h_0={\rm Id}_\cM$ and $h_1=g\circ f\circ g^{-1}$ is $\lambda$ CES. 

Let $\cL$ be a Lagrangian submanifold of $\cM$ that is H-isotopic to a graph such that $\bigcup_{n\in\Z}f^n(\cL)$ is relatively compact. As $(g_t)$ is a symplectic isotopy, $g(\cL)$ is H-isotopic to a graph and we can apply Theorem \ref{Tcompactexactdiff}. We deduce that $g(\cL)$ is exact.

As $g$ is symplectic, there is a closed 1-form $\eta$ on $\cQ$ such that $[\pi^*\eta]=[g^*\lambda-\lambda]$. Then the Liouville class of $g(\cL)$ is
$$[\cL]=[g(\cL)]+[\eta]=[\eta].$$
As $h_1=g\circ f\circ g^{-1}$, the fixed point $0$ of the action of $h_1$ on the set of Liouville classes is the image by $g^*$ of the fixed point of the action by $f$ on the Liouville classes. This means that $[\ell_f]=[\eta]$
\end{proof}

{ \begin{question*} Is the hypothesis on $H$-isotopy to the zero section necessary?\end{question*}}

\section{Uniqueness}\label{Suniqueness}
{
We work on the cotangent bundle $(T^*\cQ, -d\lambda)$ of a closed orientable manifold.\\                        
 Viterbo introduced in the seminal paper \cite{Vit1992}, see also \cite{Vit2008}, the spectral distance $\gamma$  that is   defined on the set of H-isotopic to the zero-section Lagrangian submanifolds.\\
  We will recall the main results of this theory and apply this to prove that if two submanifolds $\cL$, $\cL'$  are H-isotopic to the zero section and if  $(\varphi_t)$  is  a CH flow of $T^*\cQ$, then
 $$ \text{either}\quad\gamma (\varphi_t(\cL), \varphi_t(\cL'))\xrightarrow[]{t\rightarrow +\infty}+\infty\quad\text{or}\quad\gamma (\varphi_t(\cL), \varphi_t(\cL'))\xrightarrow[]{t\rightarrow -\infty}+\infty.$$
 Using a recent result due to Shelukhin, \cite{Shel2020}, we will deduce that for certain manifolds $\cQ$, e.g. tori $\T^n$, there is at most one H-isotopic to the zero section submanifold whose orbit is compact and when it exists, this submanifold is in fact invariant.
\subsection{On Viterbo spectral distance \texorpdfstring{$\gamma$}{gamma}}
If $\cL$, $\cL'$ are H-isotopic to the zero section submanifolds of $T^*\cQ$, they have quadratic at infinity {\em generating functions} $S:\cQ\times \R^k\rightarrow \R$ and $S':\cQ\times \R^{k'}\rightarrow \R$.\\
We recall that a generating function $S$ for $\cL$ is such that
\begin{itemize}
\item if we use the notation $(q, \xi)\in \cQ\times \R^k$, on $\Sigma_S=\Big(\frac{\partial S}{\partial \xi}\Big)^{-1}(0)$, $\frac{\partial S}{\partial \xi}$  has maximal rank;
\item the map $j_S:\Sigma_S \hookrightarrow T^*\cQ$ defined by $j_S(q, \xi)=\frac{\partial S}{\partial q}(q, \xi)$ is an embedding and its image is $\cL$.
\end{itemize}
The generating function is {\em quadratic at infinity} is there exists a non-degenerate quadratic form $Q:\R^k\rightarrow \R$ such that outside a compact subset of $\cQ\times \R^k$, we have $S(q,\xi)=Q(\xi)$.  \\
The function $S\ominus S':M\times \R^k\times \R^{k'}\rightarrow \R$ is defined by
$$(S\ominus S')(q, \xi, \chi)=S(q, \xi)-S'(q, \chi).$$
Observe that 
$$\cL\cap\cL'=\{\frac{\partial S}{\partial q}(q,\xi); d(S\ominus S')(q, \xi, \chi)=0\}=\{\frac{\partial S'}{\partial q}(q, \chi); d(S\ominus S')(q, \xi, \chi)=0\}.$$
The function $S\ominus S'$ is not quadratic at infinity, but it satisfies conditions of Proposition 1.6. of \cite{Vit2004} that ensure that it can be replaced by such a function, which we also denote by $S\ominus S'$.
There exists a compact set $K\subset \cQ\times \R^k\times \R^{k'}$ such that
$$\forall (q, \xi, \chi)\notin K, (S\ominus S')(q, \xi, \chi)=Q(\xi, \chi)$$
where $Q$ is a non degenerate quadratic form on $\R^k\times \R^{k'}$. We denote by $m$ its index. Moreover, there exist $a, b\in\R$ such that 
$$K\cap \Big(\{ (S\ominus S')\geq b\}\cup \{ (S\ominus S')\leq a\}\Big)=\emptyset.$$
For $c\in\R$, we denote by $\cE^c$ and $\cF^c$ the sublevels 
$$\cE^c=\{ (q, \xi, \chi); (S\ominus S')(q, \xi, \chi)\leq c\}\quad\text{and}\quad \cF^c=\{ (\xi, \chi); Q(\xi, \chi)\leq c\}.$$
As ($S\ominus S')(q, \xi, \chi)$ and $Q(\xi, \chi)$ are equal on $\cE^a$ and outside $\cE^b$,  we have
$$\forall c\notin]a, b[, \cE^c=\cQ\times \cF^c
.$$
Hence, by Kunneth theorem \cite{God1971},
there is  an isomorphism 
$$K:H(\cF^b, \cF^a)\otimes H(\cQ)\rightarrow H(\cE^b, \cE^a).$$
As $Q$ is a non-degenerate quadratic form with index $m$, we have $H^p(\cF^b, \cF^a)=\{ 0\}$ for $p\not=m$ and $H^m(\cF^b, \cF^a)=\R C$ is one dimensional. We deduce an isomorphism
$$T:\R C\otimes H^*(\cQ)\rightarrow H^{*+m}(\cE^b, \cE^a).$$
Then, if $\alpha\in H^*(\cQ)$ is non-zero, 
$$c(\alpha, S\ominus S')=\inf\{ t\in [a, b], j_t^*(C\otimes\alpha)\not=0\}$$
where $j_t:(\cE^t, \cE^a)\rightarrow (\cE^b, \cE^a)$ is the inclusion. The number $c(\alpha, S\ominus S')$ is then a critical value of $S\ominus S'$ that continuously depend on $S$ and $S'$ for the uniform $C^0$ distance.\\
 Viterbo proved that $c(\alpha, S\ominus S')$ depends only on $\cL$ and $\cL'$ and not on the choice of generating functions. It is then denoted by $c(\alpha, \cL, \cL')$.\\
 If $\mu$ is the orientation class of $\cQ$, the distance  $\gamma(\cL,\cL')$ is defined by
 $$\gamma(\cL, \cL')=c(\mu, \cL, \cL')-c(1, \cL, \cL').$$

\begin{theorem}\label{Tconformgamma} 
Let $f: \cM \righttoleftarrow$  be a CES diffeomorphism that is CH-isotopic to ${\rm Id}_{T^*\cQ}$.  
 Let $\cL$, $\cL'$ be two distinct  submanifolds of $T^*\cQ$ which are  H-isotopic to the zero section, then
$$ \text{either}\quad\gamma (f^n(\cL), f^n(\cL'))\xrightarrow[]{n\rightarrow +\infty}+\infty\quad$$ $$\text{or}\quad\gamma (f^{-n}(\cL), f^{-n}(\cL'))\xrightarrow[]{n\rightarrow +\infty}+\infty.$$

\end{theorem}
\begin{corollary}\label{Cunicityinavariant}
Let $f:\cM \righttoleftarrow$  be a CES diffeomorphism that is CH-isotopic to ${\rm Id}_{T^*\cQ}$. Then there exists at most one H-isotopic to the zero section submanifold of $T^*\cQ$ that is invariant by $f$.
\end{corollary}
\begin{proof}[Proof of Theorem \ref{Tconformgamma}] This is  direct application of the following result of which we provide a proof.
\begin{lemma}\label{Lconformgamma}
Let $\cL$, $\cL'$ be two H-isotopic to the zero section   submanifolds of $T^*\cQ$. Let $(\phi_t)$ be an isotopy of exact conformal symplectic diffeomorphisms of $T^*\cQ$ such that $\phi_0={\rm Id}_{T^*\cQ}$ and $\phi_t^*\omega=a(t)\omega$. Then
$$\gamma(\phi_t(\cL), \phi_t(\cL'))=a(t)\gamma(\cL, \cL').$$
\end{lemma}\begin{proof} As the distance $\gamma$  continuously depends on the generating functions, we only need to prove the results for submanifolds $\cL$ and $\cL'$ whose intersections are all transverse. In this case, there is only a finite number of critical points and critical values for $S\ominus S'$. If $x, y\in\cL\cap\cL'$, we denote by $\Delta(x, y, \cL, \cL')$ the difference of the corresponding critical values of $S\ominus S'$, i.e.
$$\Delta(x, y, \cL, \cL')=\Big(S\circ j_S^{-1}(y)-S'\circ j_{S'}^{-1}(y)\Big)-\Big(S\circ j_S^{-1}(x)-S'\circ j_{S'}^{-1}(x)\Big).$$
Then if $\eta_1$ is a path in $\cL$ joining $x$ to $y$ and $\eta_2$ a path in $\cL'$ joining $y$ to $x$,   the difference of the two corresponding critical values of $S\ominus S'$  is
$$\Delta(x, y, \cL, \cL')=\int_{\eta_1\vee\eta_2}\lambda.$$
We can always  choose $\eta_1$ and $\eta_2$ that are homotopic with fixed ends. Then, if $\cD$ is a disc with boundary $\eta_1\vee\eta_2$, we have
$$\Delta(x, y, \cL, \cL')=\int_{\cD}\omega.$$
The  intersection points of $\phi_t(\cL)$ and $\phi_t(\cL')$ are the points $\phi_t(x)$ with $x\in\cL\cap \cL'$. For $x$, $y$ in $\cL\cap\cL'$, we have 
$$\Delta(\phi_t(x), \phi_t(y), \phi_t(\cL), \phi_t(\cL'))=\int_{\Phi_t(\cD)}\omega=a(t)\int_\cD\omega=a(t)\Delta(x, y, \cL, \cL').$$
Hence $t\mapsto \frac{1}{a(t)}\Big(c(\mu, \phi_t(\cL), \phi_t(\cL'))-c(1, \phi_t(\cL), \phi_t(\cL'))\Big)$ is a continuous map that takes its values in a fixed finite set, it has to be constant.
\end{proof}

\end{proof}
\subsection{An application of a result due to Shelukhin}
\begin{theorem}   Let $f:T^*\T^n \righttoleftarrow$  be a CES diffeomorphism that is CH-isotopic to ${\rm Id}_{T^*\T^n}$.  
Then there exists at most one H-isotopic to the zero section submanifold $\cL$ such that
$$\bigcup_{k\in\Z}f^k(\cL)\quad\text{is}\quad\text{relatively}\quad\text{compact}.$$
Hence when it exists, $\cL$ is invariant by the $f$.
\end{theorem}
\begin{proof} In \cite{Shel2020}, Shelukhin defines a notion of string-point invertible manifold.  The tori $\T^n$ are examples of such manifolds. His result implies
\begin{theorem*}[Shelukhin,\cite{Shel2020}] Let $g$ be a Riemannian metric
on $\T^n$. Then there exists a constant $C(g)$ such that for all exact Lagrangian submanifolds
$\cL_0$, $\cL_1$  contained in the unit codisk bundle $D^*(g)\subset T^*\T^n$, we have
$\gamma (\cL_0, \cL_1) \leq C(g)$.

\end{theorem*}
The Liouville vector field $Z_\lambda$ that is defined by $i_{Z_\lambda}\omega=\lambda$ satisfies 
$$L_{Z_\lambda}\omega=d\lambda=-\omega.$$
Hence its flow $(\varphi^\lambda_t)$ is conformal symplectic with $\Big(\varphi^\lambda_{t}\Big)_*\omega=e^{-t}\omega$ and even exact conformal symplectic because it preserves the zero section (and then the zero Liouville class). We have seen  in Lemma \ref{Lconformgamma} that $\varphi^\lambda_t$ alters the distance $\gamma$ up to the scaling factor $e^{-t}$. \\
Observe also that this flow is  a homothety the fiber direction: $\varphi^\lambda_t(p)=e^{-t}p$. Hence the image of the unit codisk bundle $D^*(g)$ by $\varphi_t$ is the   codisk bundle $D_{e^{-t}}^*(g)$ with radius $e^{-t}$.  \\
Let us introduce the following notation for $K\subset T^*\T^n$.
$$\delta_g(K)=\min\{r\geq 0; K\subset D^*_r(q)\}.$$
Finally, we have that for every H-isotopic to the zero section submanifolds $\cL$, $\cL'$ of $T^*\T^n$, 
\begin{equation}\label{Egammanorm}\gamma(\cL, \cL')\leq 2C(g)\max\{\delta_g(\cL), \delta_g(\cL')\}.\end{equation}
If now $\cL$ and $\cL'$ are two distinct H-isotopic to the zero section submanifolds of $T^*\T^n$ and $f: T^*\T^n \righttoleftarrow$  is a CES diffeomorphism that is CH isotopic to ${\rm Id}_{T^*\T^n}$, we deduce from Theorem \ref{Tconformgamma} that 
$$ \text{either}\quad\gamma (f^n(\cL), f^n(\cL'))\xrightarrow[]{n\rightarrow +\infty}+\infty\quad$$ $$\text{or}\quad\gamma (f^{-n}(\cL), f^{-n}(\cL'))\xrightarrow[]{n\rightarrow +\infty}+\infty.$$
By \eqref{Egammanorm}, one of the two sets 
$$\bigcup_{k\in\Z}f^k(\cL);\quad  \bigcup_{k\in\Z}f^k(\cL').$$
is not relatively compact. We deduce that there is at most one $\cL$ H-isotopic to the zero section such that  $$\bigcup_{k\in\Z}f^k(\cL)\quad\text{is}\quad\text{relatively}\quad\text{compact}.$$
When $\cL$ is H-isotopic to the zero-section, $f(\cL)$ is also $H$-isotopic to the zero-section because $f$ is CH-isotopic to ${\rm Id}_{T^*\T^n}$, see Corollary \ref{PHisotopytograph}. Moreover, the orbits of $\cL$ and $f(\cL)$ coincide. This implies that $\cL=f(\cL)$.

\end{proof}
}



\section{Examples}\label{AExamples}
\subsection{Ma\~n\'e example}\label{SecMan}  This example was introduced by Ma\~n\'e in the conservative Hamiltonian setting, \cite{Man1992}. It can be extended to the conformal symplectic setting. For every vector field $X$ of a closed manifold $\cQ$, it  provides a conformal Hamiltonian Tonelli flow of $T^*\cQ$ such that the zero section is invariant and the flow restricted to this zero section is conjugated to the flow of $X$.

Let $\cQ$ be a closed manifold endowed with a Riemannian metric, $T^*\cQ$ is endowed with its tautological $1$-form $\lambda$ and the symplectic form $\omega=-d\lambda$. We denote by $\|.\|$ the norm on the fibers of $T^*\cQ$ that is dual to the Riemanninan norm of $\cQ$ and by $p_q$ a point of $T^*\cQ$ above $q \in \cQ$.

If $X$ is a vector field on $\cQ$, we denote by $p_X$ the 1-form on $\cQ$ that is dual to $X$ via the Riemannian scalar product.  We define the Hamiltonian $$H_X(p_q)=\frac{1}{2}\| p_q+p_X(q)\|^2-\frac{1}{2}\| p_X(q)\|^2.$$

Since the zero-section $\cZ=\{ p=0\}$ is contained in the zero-energy level and is Lagrangian, $\cZ$  is invariant by the Hamiltonian flow of $H_X$. The restriction to $\cZ$ of the vector field is dual via $\omega$ to the derivative of $H$ in the fiber direction, so if we denote by $\sharp: T_q^*M\rightarrow T_qM$ the duality that is defined by the Riemannian metric, we have
$$\dot q_{|\cZ} = \sharp\big(p+p_X(q)\big)_{|\cZ}=\sharp p_X(q)=X(q).$$
Hence on the zero-section, the vector field is $X$. 

In the conformal Hamiltonian setting, we add $\alpha$ times the Liouville vector field to the Hamiltonian vector  field $X_H$ of $H$, for some $\alpha \in \R$. Since the Liouville vector field vanishes on $\cZ$, the dynamics remains conjugate to $X$.

\begin{remark}
    The global attractor may differ from the zero section. For example, $X$ may have an attractive fixed point whose unstable manifold is not contained in the $\cZ$, in which case the global attractor is not a submanifold either. 
\end{remark}

\subsection{An example of a Tonelli Hamiltonian that has an invariant Lagrangian submanifold that is not a graph}\label{secBir}

The example we are about to describe is inspired by an example of Le Calvez \cite{LeCalvez1988}.

Let $\beta>0$ be a positive number and let $\alpha\in(\beta, 2\beta)$. On $T^*\R=\R^2$, let $H$ be the quadratic Tonelli Hamiltonian 
$$H(x, y)=y^2-\beta xy.$$
Consider the sum of the Hamiltonian vector field of $H$ and of $\alpha$ times the Liouville vector field $-y\, \partial_y$:
\begin{equation}\label{EPat1}\begin{cases}
  \dot x=-\beta x+2y\\
  \dot y=(\beta-\alpha) y.
\end{cases}\end{equation}
The matrix of this linear system is $\begin{pmatrix} -\beta & 2 \\
0&\beta-\alpha
\end{pmatrix}.$ Hence  $\begin{pmatrix} 1\\0
\end{pmatrix}$ is an eigenvector for the eigenvalue $-\beta$ and $\begin{pmatrix} 1\\ \beta-\frac{\alpha}{2}\end{pmatrix}$ is an eigenvector for the eigenvalue $\beta-\alpha$. As $\alpha\in(\beta, 2\beta)$, $(0,0)$ is an attracting fixed point and the line $\R\begin{pmatrix} 1\\ 0\end{pmatrix}$ is the strong stable eigenspace.  Every solution that is not contained in an eigenspace is contained in a curve whose equation is
$$x=\frac{2}{2\beta-\alpha}y+K\vert y\vert^\frac{\beta}{\alpha-\beta}$$
where $K\not=0$, and then is not a graph if $x(0).y(0)>0$. 

\begin{figure}
    \begin{tikzpicture}
        \draw[->] (-3.5,0) -- (3.5,0) node[right] {$x$};
        \draw[->] (0,-2) -- (0,2) node[above] {$y$};
        \draw[domain=-1.7:1.7][samples=200] plot (\x-\x^3, \x);
    \end{tikzpicture}
\end{figure}

Let us choose two large real numbers $B>A>0$ and let $V:\R\rightarrow [-1, 0]$ be a function with support in $[-B, B]$ such that $V_{|[-A, A]}=-1$, $V_{[-B, -A]}$ is non-increasing and $V_{|[A, B]}$ is non-decreasing. Then we add $V(x)$ to $H(x,y)$ and the equations become
\begin{equation}\label{EPat2}\begin{cases}
  \dot x=-\beta x+2y\\
  \dot y=-V'(x)+(\beta-\alpha) y.
\end{cases}\end{equation}
As the support of $V'$ is in $[-B, -A]\cup [A,B]$, the two vector fields are equal in the complement of $([-B, -A]\cup [A, B])\times \R$. As $V'_{|[-B, -A]}\leq 0$, the orbit on the $x$-axis for $x\leq -B$  is pushed to the half plane $y>0$ and then coincides with an orbit of \eqref{EPat1} which tends to $(0, 0)$. In the same way, the orbit that coincides with the $x$-axis for $x\geq B$ tends to $(0, 0)$  at $+\infty$ with an incursion into the half-plane $y<0$. Hence the union of these two orbits and $\{ (0, 0)\}$ is an invariant curve $\Gamma$ for \eqref{EPat2} that is not a graph.

Now, let us choose $D>C>B$. Le $X:\R\rightarrow \R$ a vector field such that
\begin{itemize}
    \item $\forall x\in [-\frac{D+C}{2}, -B]\cup [B, \frac{C+D}{2}], X(x)=-\beta x$;
    \item $X(-D)=X(D)=0$ and all the derivatives of $X$ are the same at $-D$ and $D$;
    \item $(-D, -B]$ (resp. $[B, D)$) is a piece of unstable manifold of the equilibrium $-D$ (resp. $D$).
\end{itemize}
Then $X$ defines also a vector field on the circle ${\mathcal C}_D=[-D, D]/D\sim -D$. Let $H_X$ be the Hamiltonian that is associated to $X$ on $T^*\R=\R^2$ via the Ma\~n\'e construction
$$H_X(x, y)=\frac{1}{2}y(y+2X(x)).$$

Let us eventually define 
$$K(x,y)=(1-\eta (x))H_X(x,y)+\eta(x)\big(H(x,y)+V(x)\big)$$
$$=\frac{1-\eta (x)}{2}y(y+2X(x))+\eta(x)\big(y^2-\beta xy+V(x)\big),$$
where $\eta:\R\rightarrow [0, 1]$ is a bump function with support in $[-C, C]$ that is equal to 1 on $[-B, B]$. $K$ also defines a Hamiltonian function on the annulus $ {\mathcal C}_D\times \R$ and, since
$$\frac{\partial^2K}{\partial y^2}(x, y)=(1-\eta(x))+2\eta(x)\geq 1$$
hence $K$ is Tonelli.

Note the following:
\begin{itemize}
    \item $([-D, -B]\cup[B, D])\times\{ 0\}$ is in the zero level of $K$ and then is locally invariant by the Hamiltonian flow of $K$ and also by the conformal Hamiltonian flow $(\frac{\partial K}{\partial y}, -\frac{\partial K}{\partial x}-\alpha y)$;
    \item $K_{|[-B, B]\times \R}=(H+V)_{|[-B, B]\times \R}$.
\end{itemize}
Finally, the vector field $(\frac{\partial K}{\partial y}, -\frac{\partial K}{\partial x}-\alpha y)$ has an invariant curve that is not a graph, which is the union of $([-D, -B]\cup[B, D])\times\{ 0\}$ and the part of $\Gamma$ that is between $x=-B$ and $x=B$.

\appendix

\section{Yomdin's inequality}
\label{AYomdin}

Let $\cL$ be a a compact Riemannian $C^r$ manifold, $\cS \subset \cL$ be a
compact $C^r$ submanifold of dimension $s$ and $f : \cL \righttoleftarrow$ be a { $C^r$-diffeomorphism}
($r \geq 1$). (The general statement does not require $f$ to be invertible.)

Define
the logarithmic volume growth of $f_{|\cS}$ as
\[\logvol(f_{|\cS}) = \limsup_{n \to +\infty} \frac{1}{n}  \log \left| \vol(f^{n}(\cS)\right|,\]
where $\vol$ is the $s$-dimensional Riemannian volume,
and
\[\rad (Df) = \limsup_{n \to +\infty} \|Df^n\|_\infty^{1/n}, \quad \|Df\|_\infty = \sup_x \|Df_x\|.\]

\begin{theorem}[Yomdin~\cite{Yom1987}, Gromov~\cite{Gromov1987}]
    \[\logvol (f_{|\cS}) \leq \ent (f) +
    \log^+\left(\rad(Df)^{s/r}\right)\]
    (where $\log^+ t = \max (0, \log t)$).
    
    In particular, if $\cL$ and $f$ are smooth,
    \[\logvol (f_{|\cS}) \leq \ent (f_{|\cS}) \leq \ent f.\]  
\end{theorem}

\section{Conformal Dynamics are exact} \label{Aconfexact}
We assume that $(\cM, \omega=-d\lambda)$ is an exact symplectic manifold. We prove that every conformal dynamics is symplectically conjugate to a CES dynamics.

{ Our first result explains that every conformal dynamics on an exact symplectic manifold is exact conformal with respect to some primitive of the symplectic form. 

}

  \begin{proposition}\label{PLiouvillembedding}
Let $f: \cM \righttoleftarrow$ be a (CS) diffeomorphism that is homotopic to ${\rm Id}_{\cM}$  and such that $f^*\omega=a\omega$. Then there exists a primitive $\lambda_1$ of $-\omega$, namely 
$$\lambda_1=\frac{1}{1-a}(\lambda-f^*\lambda)$$
such that $f$ is $\lambda_1$ CES. Hence is $j:\cS\hookrightarrow\cM$ is an isotropic embedding such that $j(\cS)$ is $f$ invariant, $j(\cS)$ is $\lambda_1$ exact.
\end{proposition}
\begin{proof}[Proof of Proposition \ref{PLiouvillembedding}]

We denote $\eta=f^*\lambda-a\lambda$. Then $d\eta=-f^*\omega+a\omega=0$ and so $\eta$ is closed.
Observe that $$\lambda-\frac{1}{1-a}\eta=\frac{1}{1-a}\big(\lambda-f^*\lambda)=\lambda_1,$$ 
so $\lambda_1$ is a primitive of $-\omega$.

We have
$$f^*\lambda_1-a\lambda_1=\eta-\frac{1}{1-a}\big(f^*\eta-a\eta\big).
$$
Because $f$ is homotopic to ${\rm Id}_\cM$, $f^*\eta-\eta$ is exact and
$$f^*\lambda_1-a\lambda_1=\frac{1}{1-a}\big( \eta-f^*\eta\big).
$$
is exact. The conclusion comes from Proposition \ref{Pinvalagrexact}  for the 1-form $\lambda_1$ instead of $\lambda$.

\end{proof}
\begin{proposition}\label{PCES}
Let $f: \cM \righttoleftarrow$ be a conformal symplectic  diffeomorphism that is homotopic to ${\rm Id}_\cM$ and such that  $f^*\omega=a\omega$ with $a>0$ and $a\not=1$. Then $\eta=f^*\lambda-a\lambda$ is a closed 1-form, there exists a symplectically isotopic to ${\rm Id}_\cM$  diffeomorphism $g: \cM \righttoleftarrow$ such that $g^*\lambda-\lambda +\frac{1}{1-a}\eta$ is exact and then $g\circ f\circ g^{-1}$ is $\lambda$ CES.
\end{proposition}
\begin{proof}
We denote $\eta=f^*\lambda-a\lambda$. Then $d\eta=-f^*\omega+a\omega=0$ and so $\eta$ is closed.  We denote by $\lambda_1$ the primitive of $\omega$ that was defined in  Proposition \ref{PLiouvillembedding}.
\begin{lemma}\label{L2}
There exists a symplectic vector field $X$ with flow $(g_t)$ such that $g_1^*\lambda-\lambda_1$ is exact.
\end{lemma}
\begin{proof}
We consider the vector field $X$ that is defined by $i_X\omega=\frac{1}{1-a}\eta$. { As $\eta$ is closed, $X$ is symplectic. }

Then we have
$$L_X\lambda=-i_X\omega+d\big(i_X\lambda\big)=-\frac{1}{1-a}\eta+d\big(i_X\lambda\big).$$
{  If we denote by $[.]$ the cohomology class, this gives
$$[L_X\lambda]=-\frac{1}{1-a}[\eta]$$
i.e.
$$\frac{d[g_t^*\lambda-\lambda]}{dt}=-\frac{1}{1-a}[g_t^*\eta].$$
We deduce that for all $t$ we have
$g_t^*\lambda-\lambda+\frac{t}{1-a}\eta$ is exact. In particular, $g^*_1\lambda-\lambda_1$ is exact.}


\end{proof}
We now consider $F=g_1\circ f\circ g_1^{-1}$. We have
$$F^*\lambda=\big(g_1^{-1}\big)^*\circ f^*\circ g_1^*(\lambda)=\big(g_1^{-1}\big)^*\circ f^*\big( \lambda_1 +\nu_1\big)$$
where $\nu_1$ is exact by lemma \ref{L2}. By Proposition \ref{PLiouvillembedding},  $\nu_2=f^*\lambda_1-a\lambda_1$ is exact and we have 
$$F^*\lambda=\big(g_1^{-1}\big)^* \big( a\lambda_1+\nu_2 + f^*\nu_1\big)=a\lambda+\big(g_1^{-1}\big)^{*}\big(-a\nu_1+\nu_2 + f^*\nu_1\big).$$
\end{proof}

\begin{proposition}\label{PCH}
Let $X$ be a conformal symplectic vector field on $\cM$ such that $L_X\omega=\alpha \omega$ with $\alpha\in\R^*$.
The 1-form $\xi=i_X\omega+\alpha\lambda$ is closed.
There exists a symplectically isotopic to ${\rm Id}_\cM$  diffeomorphism $g: \cM \righttoleftarrow$  such that $g^*\lambda-\lambda +\frac{1}{\alpha}\xi$ is exact. Then $g^*X$ is $\lambda$ conformal Hamiltonian.
\end{proposition}
\begin{proof} We have $d\xi=L_X\omega-\alpha \omega$ hence $\xi$ is closed.
\begin{lemma}\label{L3}
There exists a primitive $\lambda_1$ of $-\omega$, namely 
$$\lambda_1=\lambda-\frac{1}{\alpha}\xi=-\frac{1}{\alpha}i_X\omega,$$
such that $X$ is $\lambda_1$ Hamiltonian.
\end{lemma}
\begin{proof}
We have $$i_X\omega+\alpha\lambda_1=i_X\omega+\alpha\lambda-\xi=0$$
is exact.
\end{proof}
\begin{lemma}\label{L4}
There exists a symplectic vector field $Y$ with flow $(\psi_t)$ such that $\psi_1^*\lambda-\lambda_1$ is exact.
\end{lemma}
\begin{proof}
We consider the vector field $Y$ that is defined by $i_Y\omega=\frac{1}{\alpha}\xi$. { As $\xi$ is closed, $Y$ is symplectic.
}Then we have
$$L_Y\lambda=-i_Y\omega+d\big(i_Y\lambda\big)=-\frac{1}{\alpha}\xi+d\big(i_Y\lambda\big).$$
{We deduce that the flow $(\psi_t)$ of $Y$ satisfies
$$\frac{d}{dt}[\psi_t^*\lambda-\lambda]=-\frac{1}{\alpha}[\xi].$$
Hence $\psi_1^*\lambda-\lambda_1=\psi_1^*\lambda-\lambda+\frac{1}{\alpha}\xi$ is exact.
}

\end{proof}
We denote $g=\psi_1$. Let us prove that $g^*X$ is $\lambda$ conformal Hamiltonian. Because $g$ is symplectic, we have
$$i_{g^*X}\omega=g_{*}\big( i_X\omega\big)=g_{*}(\xi-\alpha\lambda).$$
Because $g^*\lambda-\lambda_1$ is exact, $g_*\big( \xi-\alpha\lambda\big)+\alpha\lambda$ is exact { and 
$i_{g^*X}\omega+\alpha\lambda$ is exact and so $g^*X$ is conformal Hamiltonian.}

\end{proof}

\bibliographystyle{amsplain}

\end{document}